\documentclass[reqno, 12pt]{amsart}

\usepackage[margin=1.2in]{geometry}
\usepackage{amscd,amssymb, amsmath}
\usepackage{graphicx}   

\usepackage[english]{babel}

\newcounter{thmcounter}

\newtheorem{theorem}[thmcounter]{Theorem}
\newtheorem{proposition}[thmcounter]{Proposition}
\newtheorem{lemma}[thmcounter]{Lemma}

\theoremstyle{definition}

\newcommand\id{\mathrm{id}}
\newcommand\Int{\mathrm{Int}}

\newcommand\DiffM{\mathcal{D}(M)}
\newcommand\DiffIdM{\mathcal{D}_{\id}(M)}

\newcommand\Stabf{\mathcal{S}(f)}
\newcommand\Stabfpr{\mathcal{S}'(f)}

\newcommand\StabIdf{\mathcal{S}_{\id}(f)}
\newcommand\Orbf{\mathcal{O}(f)}
\newcommand\Orbff{\mathcal{O}_{f}(f)}

\newcommand\RRR{\mathbb{R}}
\newcommand\ZZZ{\mathbb{Z}}

\newcommand\Torus{T^2} 
\newcommand\Sphere{S^2} 

\newcommand\Graphf{\Gamma(f)}
\newcommand\Aut{\mathrm{Aut}}
\newcommand\st{\mathrm{star}}

\newcommand\DiffT{\mathcal{D}(\Torus)}
\newcommand\DiffIdT{\mathcal{D}_{\id}(\Torus)}

\author{Sergiy Maksymenko, Bogdan Feshchenko}
\title{Homotopy properties of spaces of smooth functions on $2$-torus}
\address{Topology dept. Institute of mathematics of NAS of Ukraine, Tereshchenkivs'ka str. 3, Kyiv, 01601, Ukraine}
\email{maks@imath.kiev.ua, feshchenkobogdan@imath.kiev.ua}

\begin{document}
\begin{abstract}
Let $f:\Torus\to\RRR$ be a Morse function on a 2-torus, $\Stabf$ and $\Orbf$ be its stabilizer and orbit with respect to the right action of the group $\DiffT$ of diffeomorphisms of $\Torus$, $\DiffIdT$ be the identity path component of $\DiffT$,
and $\Stabfpr = \Stabf \cap \DiffIdT$.
We give sufficient conditions under which
\[ \pi_1\Orbf \ \cong \ \pi_1\DiffIdT \times \pi_0 \Stabfpr \ \equiv \ \ZZZ^2 \times \pi_0 \Stabfpr.\]
In fact this result holds for a larger class of smooth functions $f:T^2\to\mathbb{R}$ having the following property: for every
critical point $z$ of $f$ the germ of $f$ at $z$ is smothly equivalent to a homogeneous polynomial $\mathbb{R}^2\to \mathbb{R}$
without multiple factors.
\end{abstract}

\maketitle

\section{Introduction}
Let $M$ be a smooth closed oriented surface and $\DiffM$ be its groups diffeomorphisms acting from the right of the space $C^{\infty}(M,\RRR)$ of smooth functions by the following rule:
\begin{equation}\label{main-act}
 (f,h)\ \longmapsto \ \ f\circ h: M \xrightarrow{~h~} M \xrightarrow{~f~} \RRR,
\end{equation}
for $f\in C^{\infty}(M,\RRR)$ and $h \in \DiffM$.
Denote by
\begin{align*}
\Stabf &= \{f\in\DiffM \mid f\circ h = f\}, &
\Orbf &= \{f\circ h \mid h\in\DiffM\}
\end{align*}
respectively the stabilizer and the orbit of $f \in C^{\infty}(M,\RRR)$ under the action~\eqref{main-act}.
Endow the spaces $\DiffM$ and $C^{\infty}(M,\RRR)$ with strong Whitney $C^\infty$-topologies.
These topologies induce certain topologies on $\Stabf$ and $\Orbf$.
Let also $\DiffIdM$ and $\StabIdf$ be the path components of the identity map $\id_{M}$ of the groups $\DiffM$ and $\Stabf$, and let $\Orbff$ be the path component of $f$ in its orbit $\Orbf$.

Denote by $\mathrm{Morse}(M,\RRR)\subset C^{\infty}(M,\RRR)$ the subset consisting of all Morse functions, that is the functions having only non-degenerate critical points.
It is well known that $\mathrm{Morse}(M,\RRR)$ is open and everywhere dense in $C^{\infty}(M,\RRR)$, e.g.~\cite{Milnor:MorseTheory}.
Path components of $\mathrm{Morse}(M,\RRR)$ are computed in~\cite{Sharko:ProcIM:ENG:1998, Kudryavtseva:MatSb:1999, Maksymenko:CMH:2005}, and its homotopy type is described in~\cite{Kudryavtseva:MatSb:2013}.

Recall that the germs of smooth functions $f, g: (\RRR^2,0)\to (\RRR,0)$ are \emph{smoothly equivalent} at point $0\in\RRR^2$ if there exist germs of diffeomorphisms $h:(\RRR^2,0)\to(\RRR^2,0)$ and $\phi:(\RRR,0)\to (\RRR^2,0)$ such that $\phi\circ g = f\circ h$.

Let $\mathcal{F}(M,\RRR)$ be the subset of $C^{\infty}(M,\RRR)$ consisting of functions $f$ having the following property:

\noindent
{\bf Property} (L).
{\it For each critical point $z$ of $f$ its germ at $z$ is smoothly equivalent to some {\bfseries homogeneous polynomial $\RRR^2\to \RRR$ without multiple factors}.}

Notice that if $z$ is a nondegenerate critical point of a smooth function $f:M\to \RRR$ then the germ of $f$ at $z$ is equivalent to a homogeneous polynomial $\pm x^2 \pm y^2$ which obviously has no multiple factors.
Hence we have an inclusion
$$
\mathrm{Morse}(M,\RRR) \ \subset \ \mathcal{F}(M,\RRR).
$$

It is known, \cite{Poenaru:PMIHES:1970, Sergeraert:ASENS:1972}, see also~\cite[\S11]{Maksymenko:AGAG:2006}, that for functions from $\mathcal{F}(M,\RRR)$ the natural map 
\begin{equation}\label{equ:fibration_Dm_Of}
 p:\DiffM \longrightarrow \Orbf, \qquad p(h) = f\circ h,
\end{equation}
is a Serre fibration.

It is proved in~\cite{Maksymenko:AGAG:2006, Maksymenko:ProcIM:ENG:2010} that $\StabIdf$ is \emph{contractible} for every $f\in\mathcal{F}(M,\RRR)$ except for hte case when $f:S^2 \to \RRR$ is a Morse function having exactly two critical point one of which is a maximum and another one is a minimum.
In that case $\StabIdf$ is \emph{homotopy equivalent to the circle $S^1$}.

So assume that $\StabIdf$ is contractible.
Then it follows from the description of the homotopy type of groups $\DiffIdM$, see~\cite{EarleEells:DG:1970, EarleSchatz:DG:1970, Gramain:ASENS:1973}, exact sequence of homotopy groups of fibration~\eqref{equ:fibration_Dm_Of}, and from results of~\cite{Maksymenko:ProcIM:ENG:2010, Maksymenko:UMZ:ENG:2012} that $\pi_i\Orbff=\pi_i M$ for $i\geq3$, $\pi_2\Orbff=0$, and for $\pi_1\Orbff$ we have a short exact sequence
\begin{equation}\label{equ:exact_sequence}
1 \longrightarrow \pi_1\DiffIdM \xrightarrow{~p_1~} \pi_1\Orbff 
\xrightarrow{~\partial_1~} \pi_0 \Stabfpr \longrightarrow 1
\end{equation}
where $\Stabfpr = \Stabf\cap\DiffIdM$.

Notice that if $M$ is distinct from the $2$-sphere $\Sphere$ and $2$-torus $\Torus$, then the group $\DiffIdM$ is contractible, and so we get an isomorphism $\pi_1\Orbff \cong \pi_0 \Stabfpr$.

However if $M=\Sphere$ or $\Torus$ then the structure of the sequence~\eqref{equ:exact_sequence} is not understood.

The aim of the present note is to give a sufficient conditions when the sequence~\eqref{equ:exact_sequence} splits for the case $M=\Torus$, see Theorem~\ref{th:Graph_tree_Glocv_1} below.

\subsection{Graph of a smooth function}
Let $f\in\mathcal{F}(M,\RRR)$, $t\in \RRR$, and $\omega$ be the connected component of the level set $f^{-1}(t)$.
We will say that $\omega$ is \emph{critical} if it contains a critical point $f$.
Otherwise $\omega$ will be called \emph{regular}.

Consider the partition of $M$ into the connected components of level sets of $f$.
Let also $\Graphf$ be the corresponding factor space.
It is well known that $\Graphf$ has a structure of a one-dimensional CW-complex and often called the \emph{Kronrod-Reeb} graph or simply the \emph{graph} of $f$.
The vertices of $\Graphf$ are critical components of level sets of $f$, while open edges of $\Graphf$ correspond to connected components of the complement of $M$ to the union of all critical components of level sets of $f$.

Notice that $f$ can be represented as a composition
\[
 f = \phi \circ p_f: M \xrightarrow{~p_f~} \Graphf
\xrightarrow{~\phi~} \RRR,
\]
where $p_f$ is the factor map and $\phi$ is the function of $\Graphf$ induced by $f$.

\subsection{Action of $\Stabf$ on $\Graphf$}
Let $h \in \Stabf$, that is $f \circ h = f$, and so $h(f^{-1}(t))=f^{-1}(t)$ for all $t\in\RRR$.
Therefore $h$ interchanges connected components of level sets of $f$, i.e. the points of $\Graphf$.
It is easy to check that $h$ induces a certain homeomorphism $\rho(h)$ of $\Graphf$ such that the following diagram is commutative:
\begin{equation}\label{rho_h_preserves_pf}
\begin{CD}
M @>{p_f}>> \Graphf @>{\phi}>> \RRR  \\
@V{h}VV @V{\rho(h)}VV @|  \\
M @>{p_f}>> \Graphf @>{\phi}>> \RRR
\end{CD}
\end{equation}
and that the correspondence $h \mapsto \rho(h)$ is a homomorphism $\rho:\Stabf \to \Aut(\Graphf)$ into the group of all automorphisms of $\Graphf$.

Consider the group $\Stabfpr = \Stabf\cap\DiffT$ from the right part of sequence~\eqref{equ:exact_sequence}, and let
\[
 G := \rho(\Stabfpr)
\]
be its image in $\Aut(\Graphf)$.
Thus $G$ is the group of automorphisms of $\Graphf$ induced by isotopic to $\id_{M}$ diffeomorphisms from $h\in\Stabf$.
Let us emphasize that a particular isotopy between $h$ and $\id_{\Torus}$ does not necessarily consist of diffeomorphisms belonging to $\Stabf$.

Also notice that it follows from~\eqref{rho_h_preserves_pf} and the observation that the function $\phi:\Graphf\to\RRR$ is monotone of edges of $\Graphf$ that the group $G$ is finite.

Let $v$ be a vertex of $\Graphf$, and
\[ G_{v} = \{ g\in G \mid g(v)=v \}\]
be the stabilizer of $v$ with respect to $G$.
By a \emph{star} $\st(v)$ of $v$ we will mean an arbitrary connected closed $G_{v}$-invariant neighbourhood of $v$ in $\Graphf$ containing no other vertices of $\Graphf$.

Let us fix any star $\st(v)$ of $v$ and denote by 
\[
G^{loc}_{v} = \{ g|_{\st(t)} \mid g \in G_v \}
\]
the subgroups of $\Aut(\st(v))$ consisting of restrictions of elements from $G_{v}$ onto $\st(v)$.
We will call $G^{loc}_{v}$ the \emph{local stabilizer} of the vertex $v$ with respect the group $G$.
Evidently $G^{loc}_{v}$ does not depend on a particular choice of a star $\st(v)$.

The aim of this note is to prove the following two statements.
\begin{proposition}\label{pr:unique_vertex}
Let $f\in\mathcal{F}(\Torus,\RRR)$ be such that its graph $\Graphf$ is a tree.
Then there exists a unique vertex $v$ of $\Graphf$ such that the complement $\Torus\setminus p_f^{-1}(v)$ is a disjoint union of open $2$-disks.
\end{proposition}

\begin{theorem}\label{th:Graph_tree_Glocv_1}
Let $f\in\mathcal{F}(\Torus,\RRR)$ be such that its graph $\Graphf$ is a tree, and $v$ be the vertex of $\Graphf$ described in Proposition~{\rm\ref{pr:unique_vertex}}.
Suppose that the local stabilizer $G_v^{loc}$ of $v$ is a trivial group.
Then the sequence~\eqref{equ:exact_sequence} splits, and so
$$
\pi_1\Orbff \ \ \cong \ \ \pi_1\DiffIdT \times \pi_0\Stabfpr \ \ \cong \ \ \ZZZ^2 \times \pi_0\Stabfpr.
$$
\end{theorem}

\section{Proof of Proposition~\ref{pr:unique_vertex}}
Let $f\in\mathcal{F}(\Torus,\RRR)$ such that $\Graphf$ is a tree.
The following lemma is evident.
\begin{lemma}\label{lm:separating_curves}
Let $e$ be an open edge of the tree $\Graphf$, $z\in e$ be a point, and $C=p^{-1}_f(z)$ be the corresponding regular component of some level set of $f$, so $C$ is a simple closed curve in $\Torus$.
Then
\begin{enumerate}
\item $z$ divides $\Graphf$;
\item $C$ divides $\Torus$ and therefore only one of two connected components of $\Torus\setminus C$ is a $2$-disk.
\qed
\end{enumerate}
\end{lemma}

Let $e = (u_0u_1)$ be an open edge of the tree $\Graphf$, $z\in e$, and $C = p^{-1}_f(z)$ be the same as in Lemma~\ref{lm:separating_curves}.
For $i=0,1$ denote by $T_{zu_i}$ the closure of those connected component of $\Graphf\setminus z$ which contains the point $u_i$.
Put
\[X_i = p^{-1}_f(T_{zu_i}).\] 
Then by Lemma~\ref{lm:separating_curves} exactly one of two subsurfaces either $X_0$ or $X_1$ is a $2$-disk.
Let us orient the edge $e$ from $u_0$ to $u_1$ whenever $X_0$ is a $2$-disk and from $u_1$ to $u_0$ otherwise.

Then each edge of $\Graphf$ obtains a canonical orientation and so $\Graphf$ is a \emph{directed tree}. 

\begin{lemma}\label{lemma2}
For each vertex $u$ of the directed tree $\Graphf$ there exists at most one edge going from $u$.
\end{lemma}
\begin{proof}
Suppose that there are two edges going from $u$ and finishing at vertices $v$ and $v'$ respectively.
Choose arbitrary points $z_0\in(uv)$ and $z_1\in(uv')$ and denote
\begin{align*}
 A &= p_f^{-1}(T_{z_0 u}), &
 A' &= p_f^{-1}(T_{z_0 v_0}), &
 B &= p_f^{-1}(T_{z_1 u}), &
 B'&= p_f^{-1}(T_{z_1 v_1}),
\end{align*}
see Figure~\ref{fig:disks}.
By definition of orientation of edges, $A$ and $B$ are $2$-disks.
Moreover, since $\Torus = A \cup A' = B \cup B'$,
\begin{equation}\label{equ:A_Bpr__B_Apr}
 A' \,\subset\, B, \qquad \qquad
 B' \,\subset\, A,
\end{equation}
and the intersections $A\cap A' = p_f^{-1}(z_0)$ and $B\cap B' = p_f^{-1}(z_1)$ are simple closed curves, it follows that each of subsurfaces $A'$ and $B'$ is a torus with one hole.
But then neither $A'$ nor $B'$ can be embedded into a $2$-disk which contradicts to the inclusions~\eqref{equ:A_Bpr__B_Apr}.
Hence for each vertex $u$ of the directed tree $\Graphf$ there exists at most one edge going from $u$.
\end{proof}

\begin{figure}[ht]
\centerline{\includegraphics[width=7cm]{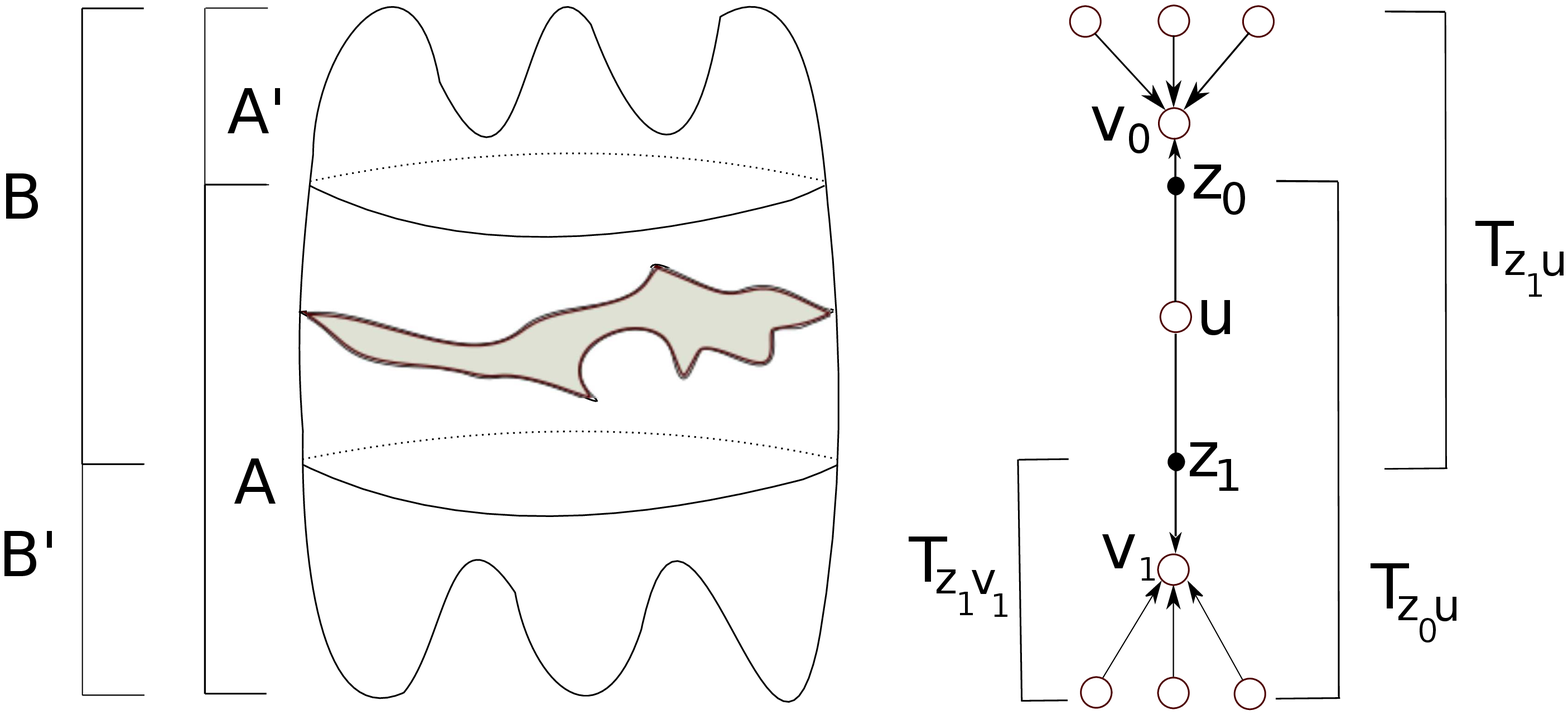}}
\caption{}
\protect
\label{fig:disks}
\end{figure}

Let $v$ be a vertex of $\Graphf$.
Notice that the complement $\Torus\setminus p_f^{-1}(v)$ is a union of $2$-disks if anf only if all edges incident to $v$ go into $v$, that is $v$ is a \emph{sink}.

Thus for the proof of Proposition~\ref{pr:unique_vertex} it suffices to prove that the oriented tree $\Graphf$ has a unique sink.
This statement is a direct consequence of the following lemma.
\begin{lemma}\label{lm:maximal_vertices}
Let $\Gamma$ be an oriented tree.
\begin{enumerate}
\item[(a)] If $\Gamma$ is finite, then it has maximal vertices.
\item[(b)] Suppose that for each vertex of $\Gamma$ there exists at most one edge going from $u$.
Then $\Gamma$ has at most one sink.
\end{enumerate}
\end{lemma}
\begin{proof}
(a) Suppose that $\Gamma$ has no sinks, so for each vertex $v$ there exist at least one edge going from $u$.
Let $v_0,{\ldots},v_{n-1},v_n$ be an arbitrary oriented path in $\Gamma$ consisting of mutually distinct vertices.
Since the edge $(v_{n-1}v_n)$ goes into $v_n$, it follows from Proposition~\ref{pr:unique_vertex} that there exists an edge $(v_nv_{n+1})$ going from $v_n$.
Notice that $v_{n+1}\neq v_i$, $i=0,\ldots,n$, otherwise $v_0,{\ldots},v_n,v_{n+1}$ would be a cycle in the tree $\Gamma$ which is impossible.
Therefore every oriented path in $\Gamma$ can be extended to a longer one which contradicts to a finiteness of $\Gamma$.
Hence $\Gamma$ must have sinks.

(b) Suppose that $\Gamma$ has two sinks $v_1$ and $v_2$ and let $\gamma: e_0,\ldots,e_k$ be a unique path connecting $v_1$ and $v_2$.
Since the edges $e_0$ and $e_k$ go into $v_1$ and $v_2$ respectively, it easily follows that for at least one vertex $u$ of $\gamma$ the edges incident to it goes from $u$, which is impossible due to the assumption, see Figure~\ref{fig:path_v1_v2}.
Hence $\Gamma$ has at most one sink.
\end{proof}

\begin{figure}[ht]
\center{\includegraphics[width=4cm]{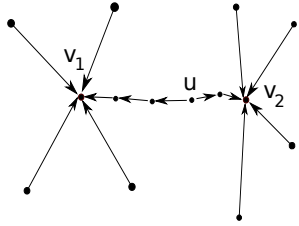} }
\caption{Path betwee $v_1$ and $v_2$}\label{fig:path_v1_v2}
\end{figure}

Now existence of a sink in $\Graphf$ follows from (a) of Lemma~\ref{lm:maximal_vertices} and its uniqueness from (b).
Proposition~\ref{pr:unique_vertex} is completed.

\section{Proof of Theorem~\ref{th:Graph_tree_Glocv_1}} 
Let $f\in\mathcal{F}(\Torus,\RRR)$ be such that $\Graphf$ is a tree, and $v$ be a unique maximal vertex of $\Graphf$ described in Proposition~\ref{pr:unique_vertex}.  
Suppose that $G_v^{loc} = 1$.
We should prove that the sequence
\begin{equation}\label{equ:exact_sequence_for_T2}
1 \longrightarrow \pi_1\DiffIdT \xrightarrow{~p_1~} \pi_1\Orbff 
\xrightarrow{~\partial_1~} \pi_0 \Stabfpr \longrightarrow 1
\end{equation}
splits.

Notice  that due to~\cite[Lemma~2.2]{Maksymenko:AGAG:2006} the image $p_1( \pi_1\DiffIdT )$ is contained in the center of the group $\pi_1\Orbff$.
Therefore for splitting of~\eqref{equ:exact_sequence_for_T2} it suffices to construct a section $s:\pi_0\Stabfpr \to \pi_1\Orbff$, that is a homomorphism such that $\partial_1\circ s = \id$.

First we recall the construction of the boundary homomorphism $\partial_1$.
Let $\omega_t$ be a loop in $\Orbff$, that is a continuous map $\omega:[0,1]\to\Orbff$ such that $\omega_0 = \omega_1$.
As $p:\DiffT\to \Orbf$ is a Serre fibration, $\omega$ can be lifted to a path in $\DiffT$.
In other words, there exists a continuous map $h:[0,1]\to\DiffT$ such that $\omega = p\circ h$, that is $\omega_t = p(h_t) = f\circ h_t$ for all $t\in[0,1]$.
Then by definition $\partial_1(\omega) = [h_1]$, where $[h_1]$ is the class of $h_1$ in $\pi_0\mathcal{S}'(f)$.

Thus if $h\in\mathcal{S}'(f)$ and $h:[0,1]\to\DiffT$ is a path such that $h_0 = \id$ and $h_1 = h$, then $\omega_t = f\circ h_t$ is a loop in $\Orbff$ such that $\partial_1(\omega) = h$.

Now Theorem~\ref{th:Graph_tree_Glocv_1} is a consequence of the following lemma.
\begin{lemma}\label{lm:conseq_Gloc_id}
Let $v$ be a unique sink of $\Graphf$, $(vu)$ be any open edge of $\Graphf$ incident to $v$, $z\in(vu)$ be a point, and $C = p_f^{-1}(z)$ be the corresponding simple closed curve on $\Torus$.
If the group $G_v^{loc}$ is trivial, then the following statements hold true.
\begin{enumerate}
\item[(i)]
Let $h\in\Stabfpr$.
Then $h(C)=C$ and there exists an isotopy $h_t:\Torus\to\Torus$, $t\in[0,1]$, such that
\begin{align}\label{equ:homotopy_h_id}
h_0 &=\id_{\Torus}, &
h_1 &= h, &
h_t(C) &= C, \ \forall t\in[0,1].
\end{align}
\item[(ii)]
If $\{h_t'\}$ is any other isotopy satisfying~\eqref{equ:homotopy_h_id}, then the paths $\{h_t\}$ and $\{h_t'\}$ are homotopic in $\DiffT$ relatively their ends.
In particular, the loops $\{f\circ h_t\}$ and $\{f\circ h_t'\}$ represent the same element of $\pi_1\Orbff$.
Denote that element by $s(h)$. 
\item[(iii)] The map $s:h\longmapsto s(h)$ is a homomorphism $s:\pi_0\Stabfpr\to\pi_1\Orbff$ such that $\partial_1\circ s = \mathrm{id}$, i.e. a section of $\partial_1$.
Hence the sequence~\eqref{equ:exact_sequence_for_T2} splits.
\end{enumerate}

\end{lemma}
\begin{proof}
(i) 
We need the following lemma, see also~\cite{KudryavtsevaFomenko:DANRAN:2012}.
\begin{lemma}\label{lm:isotopy_to_id}
Let $M$ be a smooth compact surface, $f\in\mathcal{F}(M,\RRR)$, $\Graphf$ be the graph of $f$, $\rho:\Stabf \to \Aut(\Graphf)$ be the action homomorphism of $\Stabf$ on $\Graphf$, $v$ be any vertex of $\Graphf$, $\st(v)$ be any star of $v$ in $\Graphf$, and $N = p_f^{-1}(\st(v))$.
Let also $h\in\Stabfpr$ and $\rho(h): \Graphf \to \Graphf$ be the corresponding automorphism of $\Graphf$ induced by $h$.
Suppose that $\rho(h)(v)=v$ and $\rho(h)|_{\st(v)} = \id$.
Then there exists an isotopy $g_t:\Torus\to \Torus$, $t\in[0,1]$, such that 
 \begin{enumerate}
   \item $g_0 = h$;
   \item $g_t\in\Stabfpr$;
   \item $g_1$ is fixed on $N$;
   \item $\rho(h) = \rho(g_t) = \id$ for each $t\in[0,1]$.
  \end{enumerate}
In particular, $[h] = [g_t]\in\pi_0\Stabfpr$.
\end{lemma}
\begin{proof}
Let $V = p_f^{-1}(v)$ be the critical component of the critical level set of $f$ corresponding to $v$.
Then $V$ is a finite graph embedded into $M$ and it follows from $\rho(h)(v)=v$ that $h(V)=V$.
Since $h$ is isotopic to $\id_{\Torus}$ and trivially acts on $\st(v)$, it follows from~\cite[Theorem~7.1]{Maksymenko:AGAG:2006} that $h$ preserves each edge $e$ of $V$ and keeps its orientation.
Now existence of an isotopy satisfying (1)-(4) follows from~\cite[Lemmas~6.4 and 4.14]{Maksymenko:AGAG:2006}.
\end{proof}

Let us prove (i).
Not loosing generality one can assume that there are two stars $\st_1(v)$ and $\st(v)$ of $v$ such that $z\in \st_1(v) \subset \Int(\st(v))$, where $\Int(\st(v))$ is the interior of $\st(v)$ in $\Graphf$. 
Hence if we put $N_1 = p_f^{-1}(\st_1(v))$ and $N=p_f^{-1}(\st(v))$, then $N_1\subset \Int(N)$.

Let $h\in\Stabfpr$ and $g_t:\Torus\to \Torus$, $t\in[0,1]$, be an isotopy satisfying (1)-(4) of Lemma~\ref{lm:isotopy_to_id}.
Then it follows from (3) that $\rho(g_t)(z)=z$, whence $g_t(C) = C$ for all $t\in[0,1]$.
Since $g_1$ is fixed on $N$ and the complement $\Torus\setminus N_1$ consists only of 2-disks, we see that $g_1$ is isotopic to $\id_{\Torus}$ via an isotopy fixed on $N_1$ and therefore on $C$.

Hence $h$ is isotopic to $\id_{\Torus}$ via an isotopy leaving $C$ invariant.

\medskip 

(ii) 
Let us recall the following well known statement.
\begin{lemma}\label{lm:omega}
Let $\omega:\Torus\times[0,1]\to \Torus$ be a loop in $\mathcal{D}_{\mathrm{id}}(\Torus)$, that is an isotopy $\omega_0 = \omega_1 = \mathrm{id}_{\Torus}$.
Let also $q\in \Torus$ be a point and $\omega_q:\{q\}\times[0,1]\to \Torus$ be a loop in $\Torus$ given by the formula: $\omega_q(t) = \omega(q,t)$.
Then the loop $\omega$ is null-homotopic in $\mathcal{D}_{\mathrm{id}}(\Torus)$ if and only if $\omega_q$ is null-homotopic in $\Torus$.
\end{lemma}
\begin{proof}
Since $\Torus$ is a connected Lie group, it acts on itself by right shifts being diffeomorphisms.
This action induces an embedding $i:\Torus\hookrightarrow\mathcal{D}_{\mathrm{id}}(\Torus)$.
It is well known, \cite{EarleEells:DG:1970, Gramain:ASENS:1973}, that $i$ is a homotopy equivalence.
In particular, the induced homomorphism $i^*:\pi_1\Torus\to \pi_1\mathcal{D}_{\mathrm{id}}(\Torus)$ is an isomorphism.
This easily implies that $i^*([\omega_q]) = [\omega]$ and so $\omega$ is null homotopic in $\DiffT$ if and only if $\omega_q$ is null-homotopic in $\Torus$.
\end{proof}

Let $\alpha = \{h_t\}$ and $\beta = \{h_t'\}$ be any two paths satisfying~\eqref{equ:homotopy_h_id} and $D$ be a 2-disk in $\Torus$ bounded by $C$.
Consider the loop $\omega = \alpha\beta^{-1}$ in $\DiffIdT$.
Since $\omega(C\times t)=C$, $t\in[0,1]$, we obtain that $\omega(D\times t) = D$.
Therefore for each $q\in D$ the loop $\omega_q: \{q\}\times [0,1] \to \Torus$ is null-homotopic in $\Torus$.
Hence by Lemma~\ref{lm:omega} the loop $\omega$ is null-homotopic in $\DiffIdT$, that is $\alpha$ and $\beta$ are homotopic relatively their ends.

\medskip

(iii) Let $\{h_t\}$ and $\{h_t'\}$ be paths in $\Stabfpr$ satisfying~\eqref{equ:homotopy_h_id}.
Consider the following path
$$
g_t = \begin{cases}
       h_{2t}, & t\in[0,\tfrac{1}{2}],\\
       h\circ h_{2t-1}', & t\in[\tfrac{1}{2},1]
     \end{cases}
$$
in $\mathcal{D}_{\mathrm{id}}(\Torus)$ and the corresponding loop
$$
f\circ g_t = \begin{cases}
       f\circ h_{2t}, & t\in[0,\tfrac{1}{2}],\\
       f\circ h\circ h_{2t-1}' = f\circ h_{2t-1}' , & t\in[\tfrac{1}{2},1]
     \end{cases}
$$
in $\Orbff$.
Then by definition of the multiplication in $\pi_1\Orbff$ we have that
$$
[\{f\circ h_t\}]\cdot[\{f\circ h_t'\}] = [\{f\circ g_t\}].
$$
On the other hand, $g_1 = h\circ h'$ and $g_t(C) = C$ for all $t$, that is $[\{f\circ g_t\}] = s(h\circ h')$.
Hence $s(h)\circ s(h') = s(h\circ h')$.
Lemma~\ref{lm:conseq_Gloc_id} is completed.
\end{proof}


\begin{thebibliography}{10}

\bibitem{EarleEells:DG:1970}
C.~J. Earle and J.~Eells, \emph{A fibre bundle description of teichm\"uller
  theory}, J. Differential Geometry \textbf{3} (1969), 19--43. \MR{MR0276999
  (43 \#2737a)}

\bibitem{EarleSchatz:DG:1970}
C.~J. Earle and A.~Schatz, \emph{Teichm\"uller theory for surfaces with
  boundary}, J. Differential Geometry \textbf{4} (1970), 169--185.
  \MR{MR0277000 (43 \#2737b)}

\bibitem{Gramain:ASENS:1973}
Andr{\'e} Gramain, \emph{Le type d'homotopie du groupe des diff\'eomorphismes
  d'une surface compacte}, Ann. Sci. \'Ecole Norm. Sup. (4) \textbf{6} (1973),
  53--66. \MR{MR0326773 (48 \#5116)}

\bibitem{Kudryavtseva:MatSb:1999}
E.~A. Kudryavtseva, \emph{Realization of smooth functions on surfaces as height
  functions}, Mat. Sb. \textbf{190} (1999), no.~3, 29--88. \MR{MR1700994
  (2000f:57040)}

\bibitem{Kudryavtseva:MatSb:2013}
\bysame, \emph{On the homotopy type of spaces of {M}orse functions on
  surfaces}, Mat. Sb. \textbf{204} (2013), no.~1, 79--118. \MR{3060077}

\bibitem{KudryavtsevaFomenko:DANRAN:2012}
E.~A. Kudryavtseva and A.~T. Fomenko, \emph{Symmetry groups of nice {M}orse
  functions on surfaces}, Dokl. Akad. Nauk \textbf{446} (2012), no.~6,
  615--617. \MR{3057638}

\bibitem{Maksymenko:CMH:2005}
Sergey Maksymenko, \emph{Path-components of {M}orse mappings spaces of
  surfaces}, Comment. Math. Helv. \textbf{80} (2005), no.~3, 655--690.
  \MR{MR2165207 (2006f:57028)}

\bibitem{Maksymenko:AGAG:2006}
Sergiy Maksymenko, \emph{Homotopy types of stabilizers and orbits of {M}orse
  functions on surfaces}, Ann. Global Anal. Geom. \textbf{29} (2006), no.~3,
  241--285. \MR{MR2248072 (2007k:57067)}

\bibitem{Maksymenko:ProcIM:ENG:2010}
\bysame, \emph{Functions with isolated singularities on surfaces}, Geometry and
  topology of functions on manifolds. Pr. Inst. Mat. Nats. Akad. Nauk Ukr. Mat.
  Zastos. \textbf{7} (2010), no.~4, 7--66.

\bibitem{Maksymenko:UMZ:ENG:2012}
\bysame, \emph{Homotopy types of right stabilizers and orbits of smooth
  functions functions on surfaces}, Ukrainian Math. Journal \textbf{64} (2012),
  no.~9, 1186--1203 (Russian).

\bibitem{Milnor:MorseTheory}
J.~Milnor, \emph{Morse theory}, Based on lecture notes by M. Spivak and R.
  Wells. Annals of Mathematics Studies, No. 51, Princeton University Press,
  Princeton, N.J., 1963. \MR{0163331 (29 \#634)}

\bibitem{Poenaru:PMIHES:1970}
Valentin Po{\'e}naru, \emph{Un th\'eor\`eme des fonctions implicites pour les
  espaces d'applications {$C^{\infty }$}}, Inst. Hautes \'Etudes Sci. Publ.
  Math. (1970), no.~38, 93--124. \MR{MR0474375 (57 \#14017)}

\bibitem{Sergeraert:ASENS:1972}
Francis Sergeraert, \emph{Un th\'eor\`eme de fonctions implicites sur certains
  espaces de {F}r\'echet et quelques applications}, Ann. Sci. \'Ecole Norm.
  Sup. (4) \textbf{5} (1972), 599--660. \MR{MR0418140 (54 \#6182)}

\bibitem{Sharko:ProcIM:ENG:1998}
V.~V. Sharko, \emph{Functions on surfaces. {I}}, Some problems in contemporary
  mathematics ({R}ussian), Pr. Inst. Mat. Nats. Akad. Nauk Ukr. Mat. Zastos.,
  vol.~25, Nats\=\i onal. Akad. Nauk Ukra\"\i ni \=Inst. Mat., Kiev, 1998,
  pp.~408--434. \MR{1744373 (2001j:57042)}

\end{thebibliography}

\def\cprime{$'$}
\providecommand{\bysame}{\leavevmode\hbox to3em{\hrulefill}\thinspace}
\providecommand{\MR}{\relax\ifhmode\unskip\space\fi MR }
\providecommand{\MRhref}[2]{%
  \href{http://www.ams.org/mathscinet-getitem?mr=#1}{#2}
}
\providecommand{\href}[2]{#2}

\end{document}